\documentclass[11pt,reqno]{amsart}

\usepackage{amssymb}
\usepackage{amsmath}
\usepackage{amscd}
\usepackage{amstext}
\usepackage{amsfonts}
\usepackage[all]{xy}

\newtheorem{thm}{Theorem}[section] 

\newtheorem{prop}[thm]{Proposition}
\newtheorem{cor}[thm]{Corollary}
\newtheorem{rem}[thm]{Remark}
\theoremstyle{definition}

\newtheorem{example}[thm]{Example}
\newcommand{\C}{\mathbb{C}}
\newcommand{\R}{\mathbb{R}}
\newcommand{\N}{\mathbb{N}}

\newcommand{\g}{\frak{g}}

\numberwithin{equation}{section}

\begin{document} 

\title[On splittings of Joint deformations]{On splittings of  deformations of pairs of  complex structures and holomorphic vector bundles}

\author[H. Kasuya]{Hisashi Kasuya}
\address{Graduate School of Mathematics,
Nagoya University}
\email{kasuya@math.nagoya-u.ac.jp}
\author[V. Purho]{Valto Purho}
\address{Department of Mathematics, Graduate School of Science, Osaka University, Osaka,
Japan}
\email{valto.purho@pm.me}



\begin{abstract}
We can show that the Kuranishi space of  a pair $(M,E)$ of a compact K\"ahler manifold $M$ and its flat  Hermitian vector bundle $E$ is isomorphic to the direct product of the Kuranishi space  of $M$ and  the Kuranishi space  of $E$.
We study  non-K\"ahler case.
We show that the Kuranishi space of  a pair $(M,E)$ of a complex parallelizable nilmanifold $M$ and its trivial holomorphic vector bundle $E$ is isomorphic to the direct product of the Kuranishi space  of $M$ and  the Kuranishi space  of $E$.
We  give examples of pairs $(M,E)$ of nilmanifolds $M$ with left-invariant abelian complex structures and their trivial holomorphic line bundles $E$ such that  the Kuranishi spaces of   pairs $(M,E)$ are not isomorphic to direct products of the Kuranishi spaces   of $M$ and  the Kuranishi spaces  of $E$.

\end{abstract}

\maketitle
\section{Introduction}
Let $M$ be a compact complex manifold.
By Kuranishi \cite{Ku},  we can describe the parameter space $Kur_{M}$ of all sufficiently small deformations of complex structures on $M$ as an analytic germ of an analytic space.
$Kur_{M}$ is  called the Kuranishi space of $M$.
By the analogue of  Kuranishi's  construction for deformations of holomorphic vector bundles over $M$ with a fixed complex structure, we can describe the  parameter space $Kur_{E}$ of  deformations of  all sufficiently small  holomorphic  structures on a holomorphic vector bundle $E$ (\cite{GM, GM2, Ko}).
$Kur_{E}$ is  also called 
the Kuranishi space of $E$.

In \cite{Hu}, Huang studies  deformations of a pair  $(M,E)$ of a complex manifold  $M$ and a holomorphic vector bundle  $E$ over $M$.
Extending  Kuranishi's construction  to a pair $(M,E)$, we obtain the Kuranishi space $Kur_{(M, E)}$ of a pair $(M,E)$.

The singularity of the Kuranishi space is considered  as obstructions of deformations.
We are interested in comparing  the singularity of $Kur_{(M, E)}$ with the singularity of   $Kur_{M}$ and  $Kur_{E}$.

\begin{prop}
Let $M$ be a compact K\"ahler manifold and $E$ be a flat Hermitian vector bundle.
Then we have
\[Kur_{(M,E)}\cong Kur_{M}\times Kur_{E}
\]
\end{prop}

We study  deformations of pairs of  non-K\"ahler complex manifolds with trivial holomorphic vector bundles.
A nilmanifold is a compact quotient $\Gamma \backslash G$ of a simply connected nilpotent Lie group $G$ by a discrete subgroup $\Gamma$.
It is known that $\Gamma \backslash G$ admits a K\"ahler structure if and only if $G$ is abelian and  $\Gamma \backslash G$ is a torus (\cite{BG, H}).
Thus, if $G$ is non-abelian  and admits a left-invariant complex structure, $\Gamma \backslash G$ is a non-K\"ahler complex manifold.

\begin{thm}\label{par}
Let $M=\Gamma\backslash G$ be a complex parallelizable nillmanifold   and $E$ be a trivial holomorphic  vector bundle.
Then we have
\[Kur_{(M,E)}\cong Kur_{M}\times Kur_{E}
\]
\end{thm}

 Non-parallelizable case is more complicated, we give examples of complex nilmanifods  $M=\Gamma\backslash G$ with the trivial line bundles $E$ such that 
 \[Kur_{(M,E)}\not\cong Kur_{M}\times Kur_{E}.
\]
Such examples are nilmanifolds with abelian complex structures.
Thus, on nilmanifolds with abelian complex structures, we do not have a splitting  
\[Kur_{(M,E)}\not\cong Kur_{M}\times Kur_{E}.
\]
in general.

\begin{thm}

Let  $M=\Gamma\backslash G$ be a complex  nillmanifold equipped with an abelian complex structure  and $E$ be a trivial holomorphic  vector bundle.
Then if any small deformation of the complex structure is also abelian, then
we have
\[Kur_{(M,E)}\cong Kur_{M}\times Kur_{E}.
\]
\end{thm}

\section{Kuranishi spaces of differential graded Lie algebras}
A differential graded Lie algebra (shortly DGLA) $(L^{\ast}, d,[,])$ is a $\N$-graded vector space $L^{\ast}$ equipped with a differential $d$ and a graded Lie bracket  $[,]$ satisfying the Leibniz rule.
An analytic DGLA is a normed DGLA  $(L^{\ast}, d,[,])$ whose cohomology $H^{i}(L^{\ast})$ is finite-dimensional for every $i\in \N$.

The Kuranishi space $Kur_{L}$ of  $(L^{\ast}, d,[,])$ is an analytic germ of the analytic space $K_{L}$ in $H^{1}(L^{\ast})$ at $0$ defined by the following way.
Take a Hodge decomposition $L^{i}={\mathcal H}^{i}\oplus dL^{i-1}\oplus A^{i}$ with $H^{i}(L^{\ast})\cong {\mathcal H}^{i}$.
Define the map $\delta: L^{i}\to L^{i-1}$ by the extension of  the inverse of $d : A^{i-1}\to  dL^{i-1}$ associated with the projection $L^{i}\to dL^{i-1}$ and the inclusion $A^{i-1}\subset L^{i}$.
Define  the map $F: L^1\ni x\mapsto x+\frac{1}{2}\delta [x,x]\in L^1$.
Then, the inverse of $F$ is defined on a small neighbourhood $U$ of ${\mathcal H}^{1}$ and define 
\[K_{L}=\{x\in U\vert H([F^{-1}(x), F^{-1}(x)])=0\}
\]
where $H: L^{1}\to {\mathcal H}^{1}$ is the projection.

$F^{-1}(x)$ can be written in as the Kuranishi series $F^{-1}(x)=\sum x_{i}$ such that $x_{1}=x$ and
\[x_{k}=-\frac{1}{2}\sum_{i+j=k} \delta [x_{i},x_{j}]
\]
for $k\ge 2$.
Considering each $x_{i}$ as a homogenous polynomial of $x$, 
\[ \sum H([x_{i}, x_{j}])=0
\]
are defining equations of $K_{L}$.

For a DGLA homomorphism $\phi:L^{\ast}_{1}\to L_{2}^{\ast}$ between analytic DGLAs, if $\phi$ induces isomorphisms on $0$-th and first cohomology and an injection on the second cohomology, then we have an isomorphism $Kur_{L_{1}}\cong Kur_{L_{2}}$.
We say that a morphism $\phi:L^{\ast}_{1}\to L_{2}^{\ast}$ is a quasi-isomorphism if $\phi$ induces a cohomology isomorphism.
If there is a quasi-isomorphism $\phi:L^{\ast}_{1}\to L_{2}^{\ast}$, then we have an isomorphism $Kur_{L_{1}}\cong Kur_{L_{2}}$.

\section{Kuranishi spaces of joint deformations}

Let $M$ be a compact complex manifold and $E$ a holomorphic vector bundle over $M$ equipped with a Hermitian metric $h$.
We know that the graded vector space $A^{0,\ast}(M, T^{1,0}M)$ equipped with the Dolbeault differential $\bar\partial_{T^{1,0}}$ and the  Schouten--Nijenhuis bracket $[]_{SN}$ is a DGLA which governs a deformation theory of the complex structure on $M$.
Kuranishi proves that  the Kuranishi space $Kur_{X}$ of the DGLA  $A^{0,\ast}(M, T^{1,0}M)$  is  isomorphic to a parameter space of a complete deformation family  of a complex manifold $M$ (see \cite{Ku} and \cite{GM2}).
We also know that graded vector space $A^{0,\ast}(M, End(E))$ equipped with the Dolbeault differential $\bar\partial_{End(E)}$ and the   bracket $[]_{End(E)} $ induced by the wedge product on differential forms  and the Lie bracket on  $End(E)$ is a DGLA which governs a deformation theory of the holomorphic  structure on $E$.
As similar to $Kur_{X}$,  the Kuranishi space $Kur_{E}$ of the DGLA  $A^{0,\ast}(M, End(E))$  is  isomorphic to a parameter space of a complete deformation family  of a complex manifold $M$ (see \cite{Ko} and \cite{GM2}).

Let  $D$ be the Chern connection associated with $h$ on the holomorphic vector bundle $E$.
Denote by $R\in A^{1,1}(M, End(E)) $ the curvature of $D$.
Consider the graded vector space
\[L^{\ast}(M ,E)=A^{0,\ast}(M,T^{1,0}M)\oplus A^{0,\ast}(M,  End(E)) .
\] 
Define the operator $\bar \partial_{R}: L^{\ast}(M ,E)\to L^{\ast+1}(M, E) $ by 
\[\bar \partial_{R}(\alpha , \beta)=(\bar\partial_{T^{1,0}M}\alpha , \bar\partial_{End(E)}\beta-\iota_{\alpha}R) 
\]
and the bilinear map $[,]_{D}: L^{\ast}(M ,E)\times L^{\ast}(M ,E) \to L^{\ast}(M, E) $ by 
\[\left[(\alpha_{1}, \beta_{1}), (\alpha_{2},\beta_{2})\right]_{D}=\left([\alpha_{1},\alpha_{2}]_{SN}, [\beta_{1},\beta_{2}]_{End(E)}+\iota_{\alpha_{1}}D\beta_{2}-(-1)^{pq}\iota_{\alpha_{2}}D\beta_{1}\right).
\]
Then $(L^{\ast}(M ,E), \bar \partial_{R}, [,]_{D})$ is a DGLA.
In \cite{Hu},  Huang proves that  the Kuranishi space $Kur_{(M,E)}$ of  the DGLA  $L^{\ast}(M ,E)$   is  isomorphic to a parameter space of a complete deformation family  of a pair  $(M, E)$ of  a complex manifold $M$ and a holomorphic vector bundle $E$.

\begin{prop}\label{Kah}
Let $M$ be a compact K\"ahler manifold and $E$ be a flat Hermitian vector bundle.
Then we have
\[Kur_{(M,E)}\cong Kur_{M}\times Kur_{E}
\]
\end{prop}
\begin{proof}
Since the Chern connection $D$ is flat, we have $\bar \partial_{R}=\bar\partial_{T^{1,0}M}\oplus  \bar\partial_{End(E)}$.
Consider the subset $A^{0,\ast}(M,T^{1,0}M)\oplus {\rm ker}\partial_{End(E)}$ in the DGLA $L^{\ast}(M ,E)$.
Then, for $(\alpha , \beta)\in A^{0,\ast}(M,T^{1,0}M)\oplus {\rm ker}\partial_{End(E)}$,   $\iota_{\alpha}D\beta=0$.
Hence, the subset $A^{0,\ast}(M,T^{1,0}M)\oplus {\rm ker}\partial_{End(E)}$ is a sub-DGLA in  the DGLA $L^{\ast}(M ,E)$
such that $A^{0,\ast}(M,T^{1,0}M)\oplus {\rm ker}\partial_{End(E)}$ is a direct sum of the two DGLAs $A^{0,\ast}(M,T^{1,0}M)$ and ${\rm ker}\partial_{End(E)}$.

As \cite{GM}, by the  $\partial\bar\partial$-Lermma \cite{DGMS} on flat Hermitian vector bundle ${\rm End}(E)$, the inclusion
\[
{\rm ker}\partial_{End(E)}\subset  A^{0, \ast}(M, {\rm End}(E)) 
\]
is a quasi-isomorphism and hence $Kur_{E}$ is isomorphic to the Kuranishi space of the DGLA ${\rm ker}\partial_{End(E)}$.
By this, the inclusion 
\[
A^{0,\ast}(M,T^{1,0}M)\oplus {\rm ker}\partial_{End(E)}\subset  L^{\ast}(M ,E)
\]
is also a quasi-isomorphism.
Thus the Kuranishi space $Kur_{(M,E)} $ is isomorphic to the Kuranishi space of $A^{0,\ast}(M,T^{1,0}M)\oplus {\rm ker}\partial_{End(E)}$.
Hence we have 
\[Kur_{(M,E)}\cong Kur_{M}\times Kur_{E}.
\]

\end{proof}
\begin{rem}
By more carful arguments,  the  similar statement  for deformations of  pairs of compact K\"ahler manifolds and polystable Higgs bundles with vanishing Chern classes  is proved in \cite{On}.

\end{rem}

\begin{cor}
Let $M$ be a compact K\"ahler manifold and $E$ be a flat Hermitian vector bundle.
Assume that the canonical line bundle of $M$ is holomorphically  trivial. 
Then $Kur_{(M,E)}$ is  cut out by polynomial equations of degree at most $2$.
\end{cor}
\begin{proof}
By Proposition  \ref{Kah}, we have $Kur_{(M,E)}\cong Kur_{M}\times Kur_{E}$. 
By Tian-Todorov theorem \cite{Tia, To} (see also \cite{GM2}),   $Kur_{M}$  is smooth.
As shown in \cite{GM}, the quotient map
\[
{\rm ker}\partial_{End(E)}\to H^{0, \ast}_{\partial_{End(E)}} (M, {\rm End}(E)) 
\]
is a quasi-isomorphism.
This implies that  $Kur_{E}$ is isomorphic to
\[\left\{\eta\in H^{0, \ast}_{\partial_{End(E)}} (M, {\rm End}(E)) \vert [\eta,\eta]=0\right\}.
\]
Hence, $Kur_{E}$  is  cut out by polynomial equations of degree at most $2$.
\end{proof}

\section{Nilmanifolds}

Let $G$ be a $n$-dimensional  simply connected Lie group with the Lie algebra $\g$.
A left-invariant complex structure on $G$ can  be identified with 
a  complex structure on   a real Lie algebra $\g$ i.e.  a sub-algebra $\g^{1,0}$ of the complexification $\g_{\C}$ such that 
$\g_{\C}=\g^{1,0}\oplus \g^{0,1}$ where  $\g^{0,1}=\overline{\g^{1,0}}$.
Let $p^{1,0} : \g_{\C}\to \g^{1,0}$ be the projection.
Consider $\g^{1,0}$ as a $\g^{0,1}$-module via $p^{1,0}([X, Y])$ for $X\in \g^{0,1}$, $Y\in \g^{1,0}$.
Then the cochain complex $\bigwedge (\g^{0,1})^{\ast}\otimes \g^{1,0}$ equipped with  the   bracket  induced by the wedge product and the Lie bracket on $\g^{1,0}$ is a DGLA.

Assume $G$ has a lattice $\Gamma$.
Consider the complex nilmanifold $M=\Gamma\backslash G$ with the complex structure induced by a left-invariant complex structure  on $G$.
We have the inclusions
\[\bigwedge (\g^{0,1})^{\ast}\otimes \bigwedge ^{p}(\g^{1,0})^\ast\subset A^{p,\ast}(M)
\]
 \[\bigwedge (\g^{0,1})^{\ast}\otimes \g^{1,0}\subset A^{0,\ast}(M,T^{1,0}M).\]
We say  that a left-invariant complex structure is of Calabi-Yau type  if $\bigwedge^{n} \g^{1,0\ast}$ is a trivial $\g^{0,1}$-module.
Assume that a left-invariant complex structure on $G$  is of Calabi-Yau type.
If the inclusion $\bigwedge (\g^{0,1})^{\ast}\otimes  \bigwedge ^{p}\g^{1,0\ast}\subset A^{p, \ast}(M)$ induces a cohomology isomorphism for any integer $p$, then the inclusion $\bigwedge (\g^{0,1})^{\ast}\otimes \g^{1,0}\subset A^{0,\ast}(M,T^{1,0}M)$ also induces a cohomology isomorphism (see \cite{RO}).

Assume that $G$ is nilpotent.
In this case any left-invariant complex structure on $G$  is of Calabi-Yau type (see \cite{CG}).
We call   $M=\Gamma\backslash G$ a complex nilmanifold.
As an analogous of  Nomizu's theorem (\cite{Nom}) for the de Rham cohomology of nilmanifolds,   the inclusion $\bigwedge (\g^{0,1})^{\ast}\otimes  \bigwedge ^{p}\g^{1,0\ast}\subset A^{p, \ast}(M)$ induces a cohomology isomorphism for any integer $p$  if $M$  has the structure of an iterated principal holomorphic torus bundle (\cite{CF, ROc}).
If  $(G,J)$ is a complex Lie group or  $\g^{1,0}$ is abelian, then  $M$  has the structure of an iterated principal holomorphic torus bundle.

A complex nilmanifold $M=\Gamma\backslash G$ is complex parallelizable if  and only if  $(G,J)$ is a complex Lie group.
The complex structure on a complex nilmanifold $M=\Gamma\backslash G$ is called abelian if $\g^{1,0}$ is abelian.

\begin{thm}\label{par}
Let $M=\Gamma\backslash G$ be a complex parallelizable nillmanifold   and $E$ be a trivial holomorphic  vector bundle.
Then we have
\[Kur_{(M,E)}\cong Kur_{M}\times Kur_{E}.
\]
\end{thm}

\begin{proof}
Take a triviallization    $E\cong M\times \C^{r}$.
The inclusions 
\[\bigwedge (\g^{0,1})^{\ast}\otimes \g^{1,0}\subset A^{0,\ast}(M,T^{1,0}M)\]
and 
\[\bigwedge (\g^{0,1})^{\ast}\otimes \frak{gl}_{r}( \C) \subset A^{0, \ast}(M, {\rm End}(E))  
\] 
are quasi-isomorphisms and hence $Kur_{M}$ and $ Kur_{E}$ are isomorphic to the Kuranishi spaces of $\bigwedge (\g^{0,1})^{\ast}\otimes \g^{1,0}$ and $\bigwedge (\g^{0,1})^{\ast}\otimes \frak{gl}_{r}( \C)$ respectively.

Consider the subset $\bigwedge (\g^{0,1})^{\ast}\otimes \g^{1,0}\oplus  \bigwedge (\g^{0,1})^{\ast}\otimes \frak{gl}_{r}( \C)$
 in the DGLA $L^\ast(M,E)$.
 We have $d=\bar \partial$ on  $\bigwedge (\g^{0,1})^{\ast}$.
 Thus, $\bigwedge (\g^{0,1})^{\ast}\otimes \g^{1,0}\oplus  \bigwedge (\g^{0,1})^{\ast}\otimes \frak{gl}_{r}( \C)$ is a a sub-DGLA in  the DGLA $L^{\ast}(M ,E)$
such that $\bigwedge (\g^{0,1})^{\ast}\otimes \g^{1,0}\oplus  \bigwedge (\g^{0,1})^{\ast}\otimes \frak{gl}_{r}( \C)$  is a direct sum of the two DGLAs $\bigwedge (\g^{0,1})^{\ast}\otimes \g^{1,0}$ and $ \bigwedge (\g^{0,1})^{\ast}\otimes \frak{gl}_{r}( \C)$.
Since the inclusion 
\[\bigwedge (\g^{0,1})^{\ast}\otimes \g^{1,0}\oplus  \bigwedge (\g^{0,1})^{\ast}\otimes \frak{gl}_{r}( \C)\subset L^{\ast}(M ,E)
\]
is a quasi-isomorphism,
 the Kuranishi space $Kur_{(M,E)} $ is isomorphic to the Kuranishi space of the direct sum  $\bigwedge (\g^{0,1})^{\ast}\otimes \g^{1,0}\oplus  \bigwedge (\g^{0,1})^{\ast}\otimes \frak{gl}_{r}( \C)$.
Hence we have 
\[Kur_{(M,E)}\cong Kur_{M}\times Kur_{E}.
\]

\end{proof}

\begin{cor}
Let $M=\Gamma\backslash G$ be a complex parallelizable nillmanifold   and $E$ be a trivial holomorphic  vector bundle.
Assume that the Lie algebra $\g$ of $G$ is $\nu$-step and naturally graded.
Then $Kur_{(M,E)}$ is  cut out by polynomial equations of degree at most $\nu+1$.

\end{cor}
\begin{proof}
By Theorem \ref{par}, we have $Kur_{(M,E)}\cong Kur_{M}\times Kur_{E}$. 
By \cite{ROP},  $Kur_{M}$  is  cut out by polynomial equations of degree at most $\nu$.
By \cite{Ka}, $Kur_{E}$  is  cut out by polynomial equations of degree at most $\nu+1$.
\end{proof}

\begin{thm}\label{abe}

Let  $M=\Gamma\backslash G$ be a complex  nillmanifold equipped with an abelian complex structure  and $E$ be a trivial holomorphic  vector bundle.
Then if any small deformation of the complex structure is also abelian, then
we have
\[Kur_{(M,E)}\cong Kur_{M}\times Kur_{E}.
\]
\end{thm}
\begin{proof}
By the same argument of the proof of Theorem \ref{par}, $Kur_{X}$ and $ Kur_{E}$ are isomorphic to the Kuranishi spaces of $\bigwedge (\g^{0,1})^{\ast}\otimes \g^{1,0}$ and $\bigwedge (\g^{0,1})^{\ast}\otimes \frak{gl}_{r}( \C)$ respectively.

Consider the subset $\bigwedge (\g^{0,1})^{\ast}\otimes \g^{1,0}\oplus  \bigwedge (\g^{0,1})^{\ast}\otimes \frak{gl}_{r}( \C)$
 in the DGLA $L^\ast(M,E)$.
 Then $\bigwedge (\g^{0,1})^{\ast}\otimes \g^{1,0}\oplus  \bigwedge (\g^{0,1})^{\ast}\otimes \frak{gl}_{r}( \C)$ is a a sub-DGLA in  the DGLA $L^{\ast}(M ,E)$ such that the inclusion 
\[\bigwedge (\g^{0,1})^{\ast}\otimes \g^{1,0}\oplus  \bigwedge (\g^{0,1})^{\ast}\otimes \frak{gl}_{r}( \C)\subset L^{\ast}(M ,E)
\]
is a quasi-isomorphism.
 We have $d=\bar \partial$ on  $\bigwedge (\g^{0,1})^{\ast}$. Thus $\bigwedge (\g^{0,1})^{\ast}\otimes \g^{1,0}\oplus  \bigwedge (\g^{0,1})^{\ast}\otimes \frak{gl}_{r}( \C)$ is not a direct-sum as a DGLA.
 We study Kuranishi spaces of DGLAs $\bigwedge (\g^{0,1})^{\ast}\otimes \g^{1,0}$, $\bigwedge (\g^{0,1})^{\ast}\otimes \frak{gl}_{r}( \C)$ and $\bigwedge (\g^{0,1})^{\ast}\otimes \g^{1,0}\oplus  \bigwedge (\g^{0,1})^{\ast}\otimes \frak{gl}_{r}( \C)$ precisely.

Consider a Kuranishi series $\sum x_{i}$ of the DGLA $\bigwedge (\g^{0,1})^{\ast}\otimes \g^{1,0}\oplus  \bigwedge (\g^{0,1})^{\ast}\otimes \frak{gl}_{r}( \C)$.
Write $x_{i}=\varphi_{i}+\psi_{i}$ with $\varphi_{i}\in \bigwedge (\g^{0,1})^{\ast}\otimes \g^{1,0}$ and $\psi_{i}\in \bigwedge (\g^{0,1})^{\ast}\otimes \frak{gl}_{r}( \C)$.
Then $\sum \varphi_{i}$ is a Kuranishi series of the DGLA $\bigwedge (\g^{0,1})^{\ast}\otimes \g^{1,0}$.
\cite[Theorem 4, Proposition 1]{CFP} says that for an abelian complex structure, a Kuranishi series $\sum \varphi_{i}$ of $\bigwedge (\g^{0,1})^{\ast}\otimes \g^{1,0}$ defines an abelian deformation if and only if $\sum \varphi_{i}=\varphi_{1}$ and 
\[\iota_{\varphi_{1}}\alpha=0
\]
for any $\alpha\in (\g^{0,1})^{\ast}$.
Thus, by the assumption, we have $\sum \varphi_{i}=\varphi_{1}$ and $\left[\varphi_{1}, (\g^{0,1})^{\ast}\otimes \g^{1,0}\right]_{d}=0$.
Thus $\sum \psi_{i}$ is a Kuranishi series of $ \bigwedge (\g^{0,1})^{\ast}\otimes \frak{gl}_{r}( \C)$ and defining equations of 
$Kur_{(M,E)}$ are only the equations of $ Kur_{E}$.
Hence we have 
we have
\[Kur_{(M,E)}\cong Kur_{X}\times Kur_{E}.
\]

\end{proof}

If the Lie algebra $\g$ of $G$ is isomorphic to to one of the following Lie algebras:
\[{\frak n}_{3}=(0,0,0,0,0, 12+34), {\frak n}_{8}=(0,0,0,0,0, 12), {\frak n}_{9}=(0,0,0,0,12, 14+25),
\]
then any left-invariant complex structure on $G$ is abelian (\cite[Theorem 8]{Ug}).
Thus in such case, we have
\[Kur_{(M,E)}\cong Kur_{M}\times Kur_{E}.
\]
by Theorem \ref{abe}.

\section{Non-splitting examples}
We consider real $6$-dimensional nilmanifolds  $M=\Gamma\backslash G$ with abelian complex structures.
Kuranishi spaces $Kur_{M}$ of them are computed in \cite{MPPS}.
\begin{example}
Consider the direct product $G=H_{3}(\R)\times H_{3}(\R)$  of two copies of the $3$-dimensional real Heisenberg group $H_{3}(\R)$.
Then, we have $\g=\langle X_{1}, X_{2}, Y_{1}, Y_{`2}, Z_{1}, Z_{2}\rangle $ such that $[X_{1}, Y_{1}]=Z_{1}, [X_{2}, Y_{2}]=Z_{2}$.
Define
 \[\g^{1,0}=\langle W_{1}=\frac{1}{2}(X_{1}-\sqrt{-1}Y_{1}), W_{2}=\frac{1}{2}(X_{2}-\sqrt{-1}Y_{2}), W_{3}=\frac{1}{2}(Z_{1}-\sqrt{-1}Z_{2})\rangle.\]
Then $\g^{1,0}$ is abelian.
We consider the DGLAs $L_{1}=\bigwedge (\g^{0,1})^{\ast}\otimes \g^{1,0}$, $L_{2}=\bigwedge (\g^{0,1})^{\ast}\otimes \frak{gl}_{1}( \C)$ and $L_{3}=\bigwedge (\g^{0,1})^{\ast}\otimes \g^{1,0}\oplus  \bigwedge (\g^{0,1})^{\ast}\otimes \frak{gl}_{1}( \C)$.
We see that $Kur_{L_{3}}\not \cong Kur_{L_{1}}\times Kur_{L_{2}}$.
We have
 \[
[\overline{W_{1}},W_{1}]=-\frac{1}{2}\sqrt{-1}(W_{3}+\overline{W_{3}}) \qquad{\rm and} \qquad[\overline{W_{2}},W_{2}]=\frac{1}{2}(W_{3}-\overline{W_{3}}).\]
Thus, $H^{1}(L_{1})=\langle \overline{w_{1}}\otimes W_{1}, \overline{w_{2}}\otimes W_{2}, \overline{w_{3}}\otimes W_{3}, 
\overline{w_{1}}\otimes W_{2}+\sqrt{-1}\overline{w_{2}}\otimes W_{1}\rangle$.
A Kuranishi series of $L_{1}$ is given by 
\[t_{1} \overline{w_{1}}\otimes W_{1}+t_{2} \overline{w_{2}}\otimes W_{2}+t_{3}\overline{w_{3}}\otimes W_{3}+t_{4}(\overline{w_{2}}\otimes W_{1}+\sqrt{-1}\overline{w_{1}}\otimes W_{2})-t_{3}t_{4}(\overline{w_{1}}\otimes W_{2}-\sqrt{-1}\overline{w_{2}}\otimes W_{1}).\]
We have $Kur_{L_{1}}\cong \C^4$. 
We have $H^1(L_{2})\cong H^{1}(\g^{0,1})\cong \langle \overline{w_{1}}, \overline{w_{2}}, \overline{w_{3}}\rangle $ and $Kur_{L_{2}}\cong \C^{3}$.
On the other hand, for a 
Kuranishi series
\begin{multline*}
(t_{1} \overline{w_{1}}\otimes W_{1}+t_{2} \overline{w_{2}}\otimes W_{2}+t_{3}\overline{w_{3}}\otimes W_{3}+t_{4}(\overline{w_{2}}\otimes W_{1}+\sqrt{-1}\overline{w_{1}}\otimes W_{2})+t_{3}t_{4}(\overline{w_{2}}\otimes W_{1}-\sqrt{-1}\overline{w_{1}}\otimes W_{2}), \\
 s_{1}\overline{w_{1}}+s_{2}\overline{w_{2}}+s_{3}\overline{w_{3}})
\end{multline*}
 of $L_{3}$,
 we have the defining equation
$ t_{4}s_{3}=0$.
Thus  $Kur_{L_{3}}\not \cong Kur_{L_{1}}\times Kur_{L_{2}}$.
\end{example}

\begin{example}
Consider  $G=H_{3}(\C)$ the $3$-dimensional complex Heisenberg group.
We take a real basis $X_{1}, X_{2}, X_{3}, X_{4}, Z_{1}, Z_{2}$ such that 
$[X_{1}, X_{3}]=-\frac{1}{2}Z_{1}$, $[X_{1},X_{4}]=[X_{2},X_{3}]=-\frac{1}{2}Z_{2}$, $[X_{1}, X_{3}]=\frac{1}{2}Z_{1}$.
Define \[\g^{1,0}=\langle W_{1}=X_{1}-\sqrt{-1}X_{2}, W_{2}=X_{3}+\sqrt{-1}X_{4}, W_{3}=Z_{1}+\sqrt{-1}Z_{2}\rangle.\]
Then $\g^{1,0}$ is abelian.
We consider the DGLAs $L_{1}=\bigwedge (\g^{0,1})^{\ast}\otimes \g^{1,0}$, $L_{2}=\bigwedge (\g^{0,1})^{\ast}\otimes \frak{gl}_{1}( \C)$ and $L_{3}=\bigwedge (\g^{0,1})^{\ast}\otimes \g^{1,0}\oplus  \bigwedge (\g^{0,1})^{\ast}\otimes \frak{gl}_{1}( \C)$.
We see that $Kur_{L_{3}}\not \cong Kur_{L_{1}}\times Kur_{L_{2}}$.
We have \[
[\overline{W_{1}},W_{1}]=-\sqrt{-1}(W_{3}+\overline{W_{3}}),\qquad  [\overline{W_{2}},W_{2}]=-(W_{3}-\overline{W_{3}})
\qquad {\rm and} \qquad [\overline{W_{1}},W_{2}]=-W_{3}.\]
Thus, $H^{1}(L_{1})=\langle \overline{w_{1}}\otimes W_{1}, \overline{w_{2}}\otimes W_{1}, \overline{w_{3}}\otimes W_{1}, 
\overline{w_{1}}\otimes W_{2}, \overline{w_{3}}\otimes W_{2}, \overline{w_{3}}\otimes W_{3}\rangle$.
A Kuranishi series of $L_{1}$ is given by 
\[t_{1} \overline{w_{1}}\otimes W_{1}+t_{2} \overline{w_{2}}\otimes W_{1}+t_{3}\overline{w_{3}}\otimes W_{1}+t_{4}\overline{w_{3}}\otimes W_{2}+t_{5}\overline{w_{3}}\otimes W_{2}+t_{6}\overline{w_{3}}\otimes W_{3}+t_{1}t_{6}\overline{w_{2}}\otimes W_{2}.\]
We have the defining equation $t_{3}=0$ and hence $Kur_{L_{1}}\cong \C^5$. 
We have $H^1(L_{2})\cong H^{1}(\g^{0,1})\cong \langle \overline{w_{1}}, \overline{w_{2}}, \overline{w_{3}}\rangle $ and $Kur_{L_{2}}\cong \C^{3}$.
On the other hand, for a 
Kuranishi series
\begin{multline*}
(t_{1} \overline{w_{1}}\otimes W_{1}+t_{2} \overline{w_{2}}\otimes W_{1}+t_{3}\overline{w_{3}}\otimes W_{1}+t_{4}\overline{w_{3}}\otimes W_{2}+t_{5}\overline{w_{3}}\otimes W_{2}+t_{6}\overline{w_{3}}\otimes W_{3}+t_{1}t_{6}\overline{w_{2}}\otimes W_{2}, \\
 s_{1}\overline{w_{1}}+s_{2}\overline{w_{2}}+s_{3}\overline{w_{3}})
\end{multline*}
 of $L_{3}$,
 we have the defining equation
$ t_{1}s_{3}=0$.
Thus  $Kur_{L_{3}}\not \cong Kur_{L_{1}}\times Kur_{L_{2}}$.
\end{example}

\end{document}